\documentclass[11pt]{scrartcl}
\usepackage[utf8]{inputenc}
\usepackage[T1]{fontenc}

\usepackage{manfnt}
\usepackage{mathrsfs}
\usepackage{mathpazo}

\usepackage{amsmath}
\usepackage{amsthm}
\usepackage{amsfonts}
\usepackage{amssymb}
\usepackage{graphicx}
\usepackage{enumitem}
\usepackage{hyperref}
\usepackage{pgf,tikz}
\usetikzlibrary{arrows}
\usepackage{tikz-cd}
\usepackage{bm}
\theoremstyle{plain}
\newtheorem{theorem}{Theorem}[section]
\newtheorem{proposition}[theorem]{Proposition}
\newtheorem{corollary}[theorem]{Corollary}
\newtheorem{lemma}[theorem]{Lemma}

\newtheorem{conj}{Conjecture}
\newtheorem*{conj*}{Conjecture}

\theoremstyle{definition}
\newtheorem{defn}[theorem]{Definition} 
\newtheorem*{defn*}{Definition} 

\theoremstyle{remark}
\newtheorem{remark}[theorem]{Remark}

\newtheorem{exs}[theorem]{Examples}

\newcommand{\eg}{\emph{e.g. }}
\newcommand{\ie}{\emph{i.e. }}

\newcommand{\rank}{\operatorname{rank}}
\newcommand{\vol}{\operatorname{vol}}
\newcommand{\covol}{\operatorname{covol}}
\newcommand{\ord}{\operatorname{ord}}
\newcommand{\N}{\operatorname{N}}
\newcommand{\rk}{\operatorname{rk}}

\newcommand{\Spec}{\operatorname{Spec}}
\newcommand{\Hom}{\operatorname{Hom}}
\newcommand{\GL}{\operatorname{GL}}
\newcommand{\Aut}{\operatorname{Aut}}

\newcommand{\boldell}{\pmb{\ell}}
\newcommand{\Q}{\mathbf{Q}}
\newcommand{\R}{\mathbf{R}}

\title{On slopes of isodual lattices}
\author{Renaud Coulangeon}

\usepackage{layout}
\usepackage{geometry}
\begin{document}

\maketitle
\begin{abstract}
The slope filtration of Euclidean lattices was introduced in works by Stuhler in the late 1970s, extended by Grayson a few years later, as a new tool for reduction theory and its applications to the study of arithmetic groups. Lattices with trivial filtration are called semistable, in keeping with a classical terminology. In 1997, Bost conjectured that the tensor product of semistable lattices should be semistable itself. Our aim in this work is to study these questions for the restricted class of \emph{isodual lattices}. Such lattices appear in a wide range of contexts, and it is rather natural to study their slope filtration. We exhibit specific properties in this case, which allow, in turn, to prove some new particular cases of Bost's conjecture.
\end{abstract}
\section{Introduction}

The notions of stability and slope appear in a wide range of mathematical contexts, often by analogy with the original geometric setting in which they were developed, namely the study of moduli spaces of vector bundles over curves (see \eg \cite{MR0175899,MR0184252}). 
In these various theories, one can define a \emph{canonical filtration} of an object by semistable ones, a property brought to light by Harder and Narasinham in the case of vector bundles on curves \cite{MR0364254}. A \emph{canonical polygon} is associated with this filtration, together with the sequence of \emph{slopes} of its boundary.

This formalism applies in particular to Euclidean lattices, as observed by Stuhler \cite{MR0424707,MR0447126}. The relevant notions are the height and reduced height, which in the case of an ordinary Euclidean lattice $L$ are defined as
\[ H(L)=\covol(L)=\vol(\R L/L) \ \text{ and } \ H_r(L)=H(L)^{1/\dim L}. \]
Alternatively, in keeping with the classical terminology for vector bundles over curves, one can define the \emph{degree} and \emph{slope} of $L$ as

\[ \deg L =-\log(H(L)) \] and \[ \mu(L)=-\log(H_r(L))=\dfrac{\deg L}{\rank L} . \]
These quantities are also defined for sublattices and quotients. The slopes (resp. reduced heights) of the successive quotients in the canonical filtration make up a strictly decreasing (resp. increasing) sequence of real numbers. The first term of this sequence is thus called the maximal slope $\mu_{max}$ (resp. the minimal reduced height $H_{min}$) of $L$. 

Grayson \cite{MR780079,MR870711} studied this formalism in the more general context of \emph{$\mathcal{O}_K$-lattices}, $\mathcal{O}_K$ being the ring of integers of a number field $K$. An Arakelov version of these questions was introduced by Bost in the 1990s, in terms of \emph{Hermitian vector bundles over $\Spec \mathcal{O}_K$}. More recently, Gaudron and Rémond developed a more arithmetic approach in \cite{MR3641657}, valid over any algebraic extension of $\mathbf{Q}$, in terms of \emph{rigid adelic spaces} and heights thereof. The three points of view -- $\mathcal{O}_K$-lattices, Hermitian vector bundles, rigid adelic spaces -- are equivalent when $K$ is a number field, and the above definition of (reduced) height carry over in a natural way, see section \ref{S21}. The recent text \cite{gaudron2018}, from which we borrow the approach and terminology, gives a very comprehensive account of this theory.

The slope filtration exhibits remarkable properties with respect to most of the usual algebraic operations : sum, quotient, duality. The case of tensor product is much more elusive. Formal properties of the height function, and strong analogies with similar notions in various contexts (see \cite{MR2571693,MR2872959,MR3035951})  suggested to Bost the following conjecture  :
\begin{conj}[Bost \cite{bostconj}]\label{bc}
	The minimal height of the tensor product of two rigid analytic spaces $E$ and $F$ over a number field $K$
is equal to the product of their respective minimal heights :
	\begin{equation*}
	H_{min} (E \otimes F)=	H_{min}(E)	H_{min}(F).
	\end{equation*}
\end{conj}
Such a property is known to hold in several contexts where a similar slope filtration is available (see \eg \cite{MR2571693}, where the term "tensor multiplicativity" is introduced). The proofs are most often difficult, and no really unified approach has emerged. In the case of 
Hermitian vector bundles, the conjecture has been proved for small ranks by Bost and Chen \cite{MR3035951}. Particular cases, independent of the dimension, have also been established. For instance, the conjecture is obviously true for \emph{unimodular} Euclidean lattices. Recall that a Euclidean lattice $L$ is unimodular if it coincides with its dual  $L^{\star}:=\left\lbrace y \in \mathbf{R} L ,\; \forall x \in L, \; y \cdot x \in\mathbf{Z} \right\rbrace$, where "$\cdot$" stands for the Euclidean inner product on $\mathbf{R}L$. In particular, the reduced height of a unimodular lattice is $1$, less than or equal to that of any of its sublattices. The same property holds obviously for the tensor product of two unimodular lattices, since it is also unimodular. We will see in section \ref{SS4} another interpretation of this property, which is the key of the main results in this paper.

In the light of this simple example, it seems natural to expect a special behaviour of the GS-	filtration of so-called \emph{isodual lattices}, introduced by Conway and Sloane in \cite{MR1293868} and studied by different authors from a variety of perspectives (in particular, \emph{symplectic} isodual lattices play a significant role in the study of abelian varieties , see \cite{MR1269424}). 

After reviewing the essential facts about heights and slope filtration in section \ref{S21}, we introduce in section \ref{S22} a general notion of isoduality. We observe in particular (Proposition \ref{redi}) that to prove Conjecture \ref{bc}, one can restrict to \emph{isodual} spaces. We then investigate in section \ref{SS4} the properties of the slope filtration of isodual rigid analytic spaces. The main observation, from which we derive several results, is that  the destabilizing subspace of an isodual rigid analytic space is totally isotropic with respect to a naturally defined bilinear form. The aim of section\ref{sec5} is to further reduce the proof of Conjecture \ref{bc} to the case of \emph{semistable isodual} rigid analytic spaces. Finally, as a continuation of the recent works \cite{CN3} and \cite{MR4002393}, we examine in section\ref{sec4} the influence of the automorphism group, and the associated representation, on the slope filtration of an isodual rigid analytic space. This leads us to formulate a Conjecture \ref{bciso2} -- a special case of Conjecture \ref{bc} -- which seems to be the correct "isodual" analogue of the result proved by Rémond for multiplicity free action of groups on rigid analytic spaces \cite[Théorème 1.1]{MR4002393}.

\section{Review of Hermitian bundles and semistability}\label{S21}
Let $K$ be a number field, $V(K)=V_f \cup V_{\infty}$ its set of places - finite and infinite - and $\mathcal{O}_K$ its ring of integers. Each place $v$ is associated with a normalized absolute value $\vert\cdot\vert_v$, which is the standard modulus at an Archimedean place, and is defined, at an ultrametric place $v$ associated to a prime ideal $\mathfrak{p}$, by  $\vert x\vert _v=N\mathfrak{p}^{-\ord_{\mathfrak{p}}(x)}$, where $N\mathfrak{p}=\vert \mathcal{O}_K\slash \mathfrak{p}\vert$ is the norm of the prime ideal $\mathfrak{p}$ ; the completion of $K$ with respect to this absolute value is denoted $K_v$.

If $E$ is a finite dimensional $K$-vector space, its completion $E\otimes_K K_v$ at a place $v$ is denoted $E_v$. 

A Hermitian vector bundle over Spec$({\mathcal O}_K)$ (\cite{MR1423622}), or equivalently an ${\mathcal O}_K$-lattice (\cite{MR780079, MR870711}), is the data $(L, (h_v)_{v \in V_{\infty}})$ of a finitely generated 
projective  ${\mathcal O}_K$-module $L$ together with a collection of positive definite symmetric (resp. Hermitian) forms $h_v$ on the completions $E_v$ of the $K$-vector space $E=L \otimes_{\mathcal{O}_K}K$ at real (resp. complex) Archimedean places, assumed to be \emph{invariant under complex conjugation} (see remark below). Hereafter, all Hermitian forms over a complex vector space $V$  are, by convention, 
\emph{antilinear} in the first variable, and linear in the second \[  \forall  (\lambda, \mu) \in \mathbb{C}^2,  \ \forall (x,y) \in V^2, \ h(\lambda x,\mu y)= \overline{\lambda}h(x,y)\mu. \]

\begin{remark}\label{cplex}
	A complex Archimedean place $v$ corresponds to a pair $\left\lbrace \gamma, \overline{\gamma} \right\rbrace$ of embeddings $K \hookrightarrow \mathbf{C}$, conjugated to each other. Each of them allows to identify $K_v$ with $\mathbf{C}$, giving rise to two distinct realizations of $E_v$, denoted $E_{\gamma}:=E\otimes_\gamma\mathbf{C}$ and $E_{_{\overline{\gamma}}}:=E\otimes_{_{\overline{\gamma}}}\mathbf{C}$, depending on which embedding $K \hookrightarrow \mathbf{C}$ is chosen. The "complex conjugation" is the canonical  $\mathbf{C}$-\emph{anti}linear isomorphism from $E_{\gamma}$ onto $E_{_{\overline{\gamma}}}$ defined by 
	\begin{equation}\label{cj}
	\forall x \in E, \forall \lambda \in \mathbf{C}, \  \overline{x \otimes_\gamma \lambda} := x \otimes_{_{\overline{\gamma}}} \overline{\lambda}.
	\end{equation}
	The Hermitian form $h_v$ thus consists in the data of two Hermitian forms $h_{\gamma}$ and $h_{\overline{\gamma}}$, respectively on $E_{\gamma}$ and $E_{_{\overline{\gamma}}}$, satisfying the following invariance under complex conjugation :
	\begin{equation*}
	\forall (x,y) \in E_{\gamma}\times  E_{\gamma}, \  \ h_{\overline{\gamma}}(\overline{x},\overline{y})=\overline{h_{\gamma}(x,y)}.
	\end{equation*}
\end{remark}

Equivalently, these data define a \emph{rigid adelic space} $\left( E, \left( \Vert \cdot \Vert _v \right) _{v \in V(K)}\right)$, using the terminology of \cite{MR3641657},  with local norms on each completion $E_v:= E\otimes K_v$ defined by
\begin{equation}\label{ras}
	 \forall x \in E_v\, , \ \Vert x \Vert _v =\begin{cases}
	\sqrt{h_v(x,x)} \text{ if } v \in V_{\infty}\\
		\inf\left\lbrace \vert \alpha \vert_v\, ,\ \alpha \in K_v\, , \ x \in \alpha L \right\rbrace \text{ if } v \in V_f.
\end{cases}
\end{equation}
It follows from the well-known classification of modules over Dedekind rings  (see \eg \cite[Th. 81:3]{MR1754311}) that the ${\mathcal O}_K$-module $L$ admits a \textit{pseudo-basis} $\left(\mathfrak a_i, b_i\right) _{1 \leq i \leq \ell}$, where $\mathfrak a_1, \dots \mathfrak a_{\ell}$ are fractional ideals, and $\left\lbrace b_1, \dots, b_{\ell} \right\rbrace$ is a $K$-basis of $E$ such that
\begin{equation}\label{psb} 
	L=\bigoplus_{i=1}^\ell \mathfrak a_i b_i.
\end{equation} 

By abuse of notation, we use the same letter $E$ to denote a Hermitian vector bundle (resp. a rigid adelic space) and its underlying $K$-vector space. 
If $E\neq \left\lbrace 0 \right\rbrace$, one defines its (normalized) height as
\begin{equation}\label{height}
H(E)=\left( \N\left( \prod_{i=1}^{\ell}\mathfrak a _i \right) \prod _{v \in V_\infty } 
\det(h_{v }(b_i,b_j))^{e_{v}/2}  \right)^{1/[K:\mathbb{Q}]}
\end{equation}
where $e_{v}=1$ or $2$ according to whether $v$ is real or complex, 
and we set $H(\left\lbrace 0 \right\rbrace)=1$. This corresponds, in the terminology of \cite{MR780079}, to the (normalized) volume of the $\mathcal{O}_K$-lattice $L$ defining the finite part of the rigid adelic structure, with respect with the Hermitian metrics at infinite places (it does of course not depend on the choice of a pseudo-basis for $L$, see \cite[81:8]{MR1754311}). One can check that this definition is equivalent to that of \cite{MR3641657}.

If $E\neq \left\lbrace 0 \right\rbrace$, one also define its \emph{reduced height} as
\begin{equation*}
H_r(E)=H(E)^{1/\dim E}.
\end{equation*}

If $t$ is a positive real number, we obtain a new rigid analytic space $E[t]$ by multiplying each of the Archimedean local norms at infinite places by $t$. In view of \eqref{height}, the effect on the reduced height is given by the relation 
\begin{equation}\label{scale}
H_r(E[t])=tH_r(E).
\end{equation}

Any subspace $F$ of a rigid adelic space $E$ inherits the structure of a rigid adelic space, by restricting the local norms at all places. In the language of Hermitian bundles, it amounts to replace the $\mathcal{O}_K$-lattice $L$ with $L \cap F$, and restrict to $F$ the Hermitian forms at infinite places. Consequently, one defines, for all positive integer $k$,
\begin{equation}\label{hk}
	H^{(k)}(E)=\min\limits_{\substack{0 \neq F \subset E\\ \dim F = k}} H_r(F)
\end{equation}
and 
\begin{equation}\label{hmin}
	H_{min}(E)=\min\limits_{0 \neq F \subset E} H_r(F).
\end{equation}
\begin{remark}
	Because of \eqref{scale}, one has \begin{equation}\label{hmt}
		\forall t>0, \  H_{min}(E[t])= t H_{min} (E).
	\end{equation}
\end{remark}
\medskip

To any  $K$-linear map $\sigma : E \rightarrow F$  between rigid adelic spaces one can associate a family of localised maps $\sigma_v : E\otimes K_v \rightarrow F\otimes K_v$ defined as usual by
\begin{equation}\label{sloc}
	\forall x \in E, \forall \lambda \in K_v, \ 	\sigma_v(x \otimes \lambda) = \sigma(x) \otimes \lambda 
\end{equation} 
and extended by bilinearity.
\begin{remark}
	At a complex place $v$, associated to a pair $\left\lbrace \gamma,\overline{\gamma} \right\rbrace$ of complex embeddings, there are two realizations $\sigma_{\gamma}$ and $\sigma_{\overline{\gamma}}$ of $\sigma_v$, depending on the choice of the embedding of $K$ in $\mathbf{C}$ (remark \ref{cplex}).
\end{remark}

\begin{defn}
	An \emph{isometry} between two rigid adelic spaces $E$ and $F$ over $K$ is a $K$-linear map $\sigma : E \rightarrow F$ such that the localised maps $\sigma_v : E\otimes K_v \rightarrow F\otimes K_v$ preserve the local norms for all $v \in V$. 
\end{defn}

The quotient $E/F$  of a rigid adelic space by a subspace
also inherits a canonical structure of rigid adelic space, using quotient norms (see  \cite[\S 2]{MR3641657}). To express it in terms of  Hermitian bundles, one has to consider the quotient $\mathcal{O}_K$-lattice $L / L \cap F$, and use the identification $E/F\otimes K_v \simeq F_v^{\perp_{h_v}}$ (orthogonal complement with respect to $h_v$) at infinite places to define Hermitian structures, see \cite{MR780079}.\medskip

 Similarly, the operator norms induce on the dual space $E^{\vee}=\Hom_K(E,K)$ an adelic structure which can also be viewed  as the structure induced by the $\mathcal{O}_K$-lattice  $L^{\vee}=\Hom_{\mathcal{O}_K}(L,\mathcal{O}_K)$ equipped with the Hermitian forms
\begin{equation*}
h^{ \vee}_v(y^{\vee},y^{\vee})=\sup_{0 \neq x \in E}\dfrac{\vert y^{\vee}(x)\vert_v^2}{h_v(x,x)}
\end{equation*}
at infinite places.
Note in particular that
\begin{equation}\label{dual}
H_r(E^{\vee})=H_r(E)^{-1}
\end{equation}
and 
\begin{equation}\label{scaledual}
	E[t]^{\vee}=E^{\vee}[t^{-1}] \text{ for all } t>0.
\end{equation}
As usual, the orthogonal in $E^{\vee}$ of a subspace $F$ of $E$ is defined as
\begin{equation}\label{perp}
F^{\perp}:=\left\lbrace \varphi \in E^{\vee} \mid \varphi(F)=0\right\rbrace.
\end{equation}
As one might expects, the previous notions are related through the following property :
\begin{proposition}\cite[Proposition 3.6]{MR3641657}\label{dq} For any subspace $F$ of a rigid analytic space $E$, one has $E^{\vee}\slash F^{\perp} \simeq F^{\vee}$.
\end{proposition}

The direct product $E \times F$ of two adelic rigid spaces is equipped with the local norms
\[  \Vert (x,y)\Vert_v = \begin{cases}
 	\left( \Vert x \Vert_v^2 + \Vert y \Vert_v ^2 \right)^{1/2} \text{ at infinite places}, \\  \max \left\lbrace \Vert x \Vert_v, \Vert y \Vert_v \right\rbrace \text{ at finite places.}
 \end{cases} \]
In terms of Hermitian bundles, it corresponds to the usual direct sum of $\mathcal{O}_K$-modules, endowed, at each infinite place,  with the orthogonal direct sum of the corresponding Hermitian forms.\medskip

Finally, the tensor product of adelic spaces/Hermitian bundles is defined naturally with either point of view, see \emph{loc. cit}.

\begin{remark}
	In the references \cite{MR3035951} and \cite{gaudron2018} on which we rely, the authors call Hermitian direct \emph{sum} $E\oplus F$ of rigid adelic spaces what we just defined as their direct \emph{product} $E\times F$. The reason why we chose to avoid the direct sum notation/terminology is that we think it can occasionally be misleading : for instance, if $F$ and $F'$ are subspaces of a given rigid adelic space $E$ intersecting trivially, the structure of adelic rigid spaces induced by $E$ on the subspace $F \oplus F'$ \emph{is not} their Hermitian direct sum in general.
\end{remark}
 
A key property of reduced height is that the set of subspaces of a given Hermitian bundle $E$ with minimal reduced height has a well-defined maximum $E_1$ with respect to inclusion, called the destabilizing subspace of $E$ (see see \emph{e.g.} \cite[Satz 1]{MR0424707} for a proof in the case of lattices and \cite{gaudron2018} for the general case). In other words, any rigid analytic space $E$ contains a unique subspace $E_1$ characterized by the following two properties :
\begin{enumerate}
	\item $H_r(E_1)=H_{min}(E)=\min\limits_{0 \neq F \subset E} H_r(F)$.
	\item  Any subspace $F$ of $E$ such that $H_r(F)=H_{min}(E)$ is contained in $E_1$.
\end{enumerate}

A rigid analytic space $E$ is \emph{stable} if $H_r(F) > H_r(E)$ for all proper subspace $\left\lbrace 0 \right\rbrace \subsetneq F \subsetneq E$, \emph{semistable} if $H_r(F) \geq H_r(E)$ for all subspace $F$, and \emph{unstable} if it is not semistable. 
In particular, $E$ is semistable if and only if it coincides with its destabilizing subspace $E_1$.

We denote by 
\begin{equation*} 
\left\lbrace 0 \right\rbrace=E_0 \subset E_1 \subset \dots \subset E_{\ell-1} \subset E_{\ell}=E
\end{equation*}
the Grayson-Stuhler filtration of $E$ ("GS-filtration of $E$" for short) , defined recursively as follows :
\begin{enumerate}
	\item $E_1$ is the destabilizing subspace of $E$.
	\item For $i\geq 2$, $E_i/E_{i-1}$ is the destabilizing subspace of $E/E_{i-1}$.
\end{enumerate}
An alternative characterization of this filtration is given by the following proposition :
\begin{proposition}\cite[Corollary 1.30]{MR780079}\label{gsc}
	The GS-filtration of $E$ is the unique flag $(E_i)_{0 \leq i \leq \ell}$ such that
	\begin{enumerate}
		\item $E_i\slash E_{i-1}$ is semistable for all $1 \leq i \leq \ell$, 
		\item If $\ell>1$ then $H_r\left(  E_i\slash E_{i-1}\right) < H_r\left( E_{i+1}\slash E_i \right)$  for all $1\leq i \leq \ell-1$.
	\end{enumerate}
\end{proposition}

\begin{defn}
The integer $\ell$ is called the \emph{length} of the GS-filtration of $E$. In particular, $E$ is semistable if and only if the length of its GS-filtration is equal to $1$.
\end{defn}

The uniqueness property entails two remarkable properties of the GS-filtration :  it is invariant under automorphisms (see section \ref{sec4}), and scalar extension (see \eg \cite[Proposition 19]{gaudron2018}). 

We end this section with a useful lemma, well-known to the experts, about the behaviour of $H_{min}$ with respect to quotients and products. 
\begin{lemma}\label{ds} \quad
	\begin{enumerate} 
		\item\label{ena} For any two rigid analytic spaces $\left\lbrace 0 \right\rbrace \neq F \subset E$, one has
	\[ \min(H_{min}(F),H_{min}(E/F)) \leq H_{min}(E) \leq H_{min}(F). \]
	In particular, $H_{min}(E)=H_{min}(F)$ if $H_{min}(F)\leq H_{min}(E/F))
$.	\item The minimal reduced height of the 
direct product 
of two rigid analytic spaces is given by
\[ 	H_{min}(E \times F)= \min (H_{min}(E), H_{min}(F)). \]
\end{enumerate}	
\end{lemma}
\begin{proof} A proof of the first assertion can be found in \cite[Lemma 1.9]{CN3}\label{first}. The second assertion is a direct consequence of the first one, using the exact sequences 
\[ 0 \rightarrow E  \rightarrow E \times F \rightarrow F \rightarrow 0 \text{ and } 0 \rightarrow F \rightarrow E \times F \rightarrow E \rightarrow 0. \]
\end{proof}

Regarding Conjecture \ref{bc}, the first assertion of the previous lemma has the following consequence :
\begin{corollary}\label{c1}
	Let $E$ and $F$ be rigid analytic spaces. Suppose that $F$ admits a filtration 
	\[ \left\lbrace 0 \right\rbrace=F_0 \subset F_1 \subset \dots \subset F_{t-1} \subset F_t=F \] such that \begin{enumerate}[label=(\roman*)]
		\item $H_{min}(E\otimes F_i/F_{i-1} )=H_{min}(E) H_{min}(F_i/F_{i-1}) \ \text{ for all } 1 \leq i \leq t$,
		\item $H_{min}(F_i/F_{i-1}) \leq H_{min}(F_{i+1}/F_i) \ \text{ for all } 1 \leq i \leq t-1$.
	\end{enumerate}
	Then \[ H_{min}(E\otimes F)=H_{min}(E)H_{min}(F). \]
\end{corollary}
\begin{proof}
	Recursion on $t$, using the first assertion of Lemma\ref{ds}.
\end{proof}
\begin{remark}\label{red}
One noticeable consequence of the above Corollary is that Bost's conjecture \ref{bc} is equivalent to the -- apparently weaker -- statement that the tensor product of \emph{semistable} rigid analytic spaces is itself \emph{semistable} (see \cite[p. 440]{MR3035951}, and \cite{CN3} for a discussion of this point).
\end{remark}\smallskip
 Another easy consequence of Lemma \ref{ds}	and its corollary is the following :
 \begin{corollary}\label{r2}
 	Let $F$ be a rigid analytic space of rank $2$ over a number field $K$, which is not stable. Then for any rigid analytic space $E$ over $K$, one has $H_{min}(E\otimes F)=H_{min}(E)H_{min}(F)$.
 \end{corollary}
\begin{proof}
If $F$ is unstable, its destabilizing subspace $F_1$ is one-dimensional, as well as $F/F_1$. Consequently, \[ H_{min}(E\otimes F_1)=H_{min}(E)H_{min}(F_1)\text{ and  }H_{min}(E\otimes F/F_1)=H_{min}(E)H_{min}(F/F_1), \] and one can apply Corollary \ref{c1} to conclude that $H_{min}(E \otimes F)=H_{min}(E)H_{min}(F)$.
 If $F$ is unstable but semistable,  one can "destabilize" it by an arbitrary small perturbation of the infinite components of the metric, in which case the previous argument applies, whence the conclusion since the equality $H_{min}(E \otimes F)=H_{min}(E)H_{min}(F)$ is preserved under taking limits.
\end{proof}

\section{Isodual rigid adelic spaces}\label{S22}


Recall that we have defined an \emph{isometry} between two rigid adelic spaces $E$ and $F$ over $K$ as a $K$-linear map $\sigma : E \rightarrow F$ such that the localised maps $\sigma_v : E\otimes K_v \rightarrow F\otimes K_v$ preserve the local norms for all $v \in V$.

If $L$ and $M$ are the $\mathcal{O}_K$-lattices underpinning $E$ and $F$, it is easy to see, due to the description of the local norms at finite places, that an isometry $\sigma$ is simply a $K$-linear isomorphism mapping $L$ onto $M$ and inducing Hermitian isometries at all infinite places.\medskip

More generally, one can define the notion of \emph{similarity} as follows :
\begin{defn}
	A \emph{similarity} between two rigid adelic spaces $E$ and $F$ over $K$ is a $K$-linear map  $\sigma : E \rightarrow F$ such that the maps $\sigma_v$ ($v \in V$) are similarities with respect to the local norms, and the similarity ratio is $1$ at all but finitely many places. 
\end{defn}
Again, if $L$ and $M$ are the $\mathcal{O}_K$-lattices underpinning $E$ and $F$, one easily checks that a $K$-linear map  $\sigma : E \rightarrow F$ is a similarity if $\sigma_v$ is a Hermitian similarity for every Archimedean place $v$, and if there exists a fractional ideal $\mathfrak{a}$ such that $\sigma(L)=\mathfrak{a}M$ .

 \medskip
 
This notion of similarity is relevant to our purpose, since it preserves the Grayson-Stuhler filtration of a rigid adelic space, as shows the following lemma :

\begin{lemma}\label{gssd}
	Let $E$ be a rigid adelic space, with GS-filtration
	\begin{equation*}\left\lbrace 0 \right\rbrace=E_0 \subset E_1 \subset \dots \subset E_{\ell-1} \subset E_{\ell}=E.\end{equation*}
	and $\sigma$ a similarity. Then the GS-filtration of $\sigma E$ is
	\begin{equation*}
	\left\lbrace 0 \right\rbrace=\sigma E_0 \subset \sigma E_1 \subset \dots \subset \sigma E_{\ell-1} \subset \sigma E_{\ell}=\sigma E.
	\end{equation*}
\end{lemma}
\begin{proof}
	Let $\lambda =(\lambda_v)_{v \in V}$ where $\lambda_v$ is the similarity ratio of  $\sigma_v$ \ie $\Vert \sigma_v(x)\Vert_v =\lambda_v \Vert x \Vert_v$ for all $x \in E_v=E\otimes_K K_v$, and set $N(\lambda)=\prod_{v \in V} \lambda_v$. Then  
	\[ H_r(\sigma(F))=N(\lambda)^{1/[K:\mathbb{Q}]}H_r(F) \]
	for every subspace $F$ of $E$. 
	The conclusion follows, as the scaling factor 
	)$N(\lambda)^{1/[K:\mathbb{Q}]}$ is independent of $F$ and its dimension.
\end{proof}

\begin{defn}\label{isod}
	A rigid analytic space $E$  is $\sigma$\emph{-isodual}, or simply \emph{isodual}, if there exists a similarity $\sigma : E \rightarrow E^{\vee}$. To such a similarity, one associates a $K$-bilinear form $b_{\sigma} : E\times E \rightarrow K$ defined by
	\begin{equation}\label{bsig}
	\forall (x,y)	\in E\times E\ , \ \ b_{\sigma} (x,y)=\sigma(x)(y).
	\end{equation}
\end{defn}
If $K$ is a $CM$-field, with complex conjugation $\bar{\phantom{u}}$, it can be more natural to consider instead the \emph{conjugate} dual space $\overline{E^{\vee}}$, which is the ordinary dual $E^{\vee}$ (set of $K$-linear forms) equipped with the twisted external law 
\[  \alpha \star \varphi := \overline{\alpha} \varphi, \ \ \left( \alpha  \in K, \varphi \in E^{\vee} \right). \]
 It is consistent in this case to consider spaces admitting a similarity onto their \emph{conjugate} dual instead. This leads to the following extension of definition \ref{isod} :
\begin{defn}
		A rigid analytic space $E$  over a $CM$-field $K$ is \emph{anti-}isodual, if there exists a similarity $\sigma : E \rightarrow \overline{E^{\vee}}$. 
\end{defn}

Note that in this situation, equation \eqref{bsig} defines a \emph{sesquilinear} form $b_{\sigma}$ on $E$.

\begin{defn}\label{typ}
Let $(E,\sigma)$ be an isodual rigid adelic space over a number field $K$, or an anti isodual space over a $CM$-field $K$.
We say that $E$ is 
\begin{enumerate}
	\item \emph{orthogonal} if the bilinear form $b_{\sigma}$ is symmetric,
	\item \emph{symplectic} if $b_{\sigma}$ is alternate.
	\item \emph{unitary} if $K$ is a CM-field, $E$ is anti isodual  and $b_{\sigma}$ is Hermitian.
\end{enumerate}
\end{defn}
Note that in all three cases, $b_{\sigma}$ is necessarily non-degenerate.
\begin{remark}\label{latt}
The previous definitions extend naturally the notion of \emph{isodual lattice} mentioned in the introduction, which we now recall in a slightly greater generality : suppose that $K$ is either a totally real or a CM-field, and $E$ a $K$-vector space endowed with a \emph{totally positive} definite quadratic (resp. Hermitian) form $h$, \ie $h$ is a $K$-valued quadratic or Hermitian form on $E$ such that the extensions $h_v$ of $h$ to all completions $E_v$ at infinite places are positive definite. With $\bar{\phantom{u}}$ denoting either the complex conjugation if $K$ is a $CM$-field, or the identity if $K$ is totally real, we get an isomorphism $H$ between $E$ and $\overline{E^{\vee}}$ ($=E^{\vee}$ when $K$ is totally real)  given by
\begin{equation*}
	H \colon\begin{aligned}[t]
		E &\longrightarrow \overline{E^{\vee}}\\
		x&\longmapsto h(x,\cdot)
	\end{aligned}
\end{equation*}
Any \emph{lattice} $L$ in $E$ (=full-rank finitely generated projective  ${\mathcal O}_K$-submodule of $E$)  induces a structure of rigid adelic space on $E$,
the infinite part consisting of the extensions $h_v$ of $h$ to the completions $E_v$, for $v \in V_{\infty}$. 
The preimage of $L^{\vee}=\Hom(L,\mathcal{O}_K)$ by $H$ is 
\begin{equation}\label{dualbis}
	L^{*}:= \left\lbrace y \in E \mid h\left(L,y\right) \subset {\mathcal O}_K \right\rbrace. 
\end{equation}
We say that the lattice $L$ is \emph{isodual} if there exists 
a similarity $\tau$ of the Hermitian space $(E,h)$ mapping $L$ onto $L^{*}$, 
which means that there exists $\alpha \in K$ such that \[ \forall (x,y) \in E \times E,  \ h(\tau(x),\tau(y))=\alpha h(x,y).  \]
If so, the map $\sigma=H\circ \tau$ is a similarity of rigid analytic space between $E$ and $\overline{E^{\vee}}$
and the form $b_{\sigma}$ of \eqref{bsig} is given by \[  b_{\sigma}(x,y)=h(\tau(x),y). \] 
The ratio  $\alpha$ of $\tau$ is a totally positive element in $K^{+}$, the maximal totally real subfield of $K$. 
It is then easily checked that, as a $H\circ \tau$-isodual (resp. \emph{anti} isodual) rigid analytic space, $E$ is orthogonal (resp. unitary) if and only if $\tau^2=\alpha Id$ and symplectic if and only if $\tau^2=-\alpha Id$.

\end{remark}\medskip
The last remark highlights an important class of  rigid adelic spaces, stemming from lattices in quadratic (resp. Hermitian) spaces over a totally real (resp. CM)  number field. This motivates the following definition.

\begin{defn}\label{rat}
A rigid adelic space $E$ over a number field $K$ is $K$-rational if $K$ is either a totally real or a $CM$ extension of $\mathbf{Q}$, and the symmetric (resp. Hermitian) forms $h_v$ at Archimedean places come from a $K$-valued symmetric (resp. Hermitian) form $h$ on $E$ by localization. 
\end{defn}

\begin{remark}
	Every $K$-rational rigid adelic space $E$ of dimension $2$ over a totally real or a $CM$ extension of $\mathbf{Q}$ is isodual (resp anti-isodual). Indeed, if $E$ is endowed with a totally positive definite quadratic (resp. Hermitian) form $h$ over $K$ which defines the local metrics at Archimedean places by localization, while the metrics at finite places are determined by the data of an $\mathcal{O}_K$-lattice $L=\mathfrak{a} e_1 \oplus \mathcal{O}_K e_2$ in $E$,  then  the Gram matrix of $h$ in the basis $(e_1,e_2)$ has the following shape :
	\[ \begin{pmatrix}
		a & c \\
		c & b
	\end{pmatrix} \text{ resp. } \begin{pmatrix}
		a & c \\
		\overline{c} & b
	\end{pmatrix}.  \]
The Gram matrix of $h$ in the dual basis $(e^{*}_1,e^{*}_2)$, defined by the condition that $h(e^{*}_i,e_j)=\delta_{i,j}$,  is thus given by
	
\[ 	\dfrac{1}{ab-c^2}\begin{pmatrix}
		b & -c \\
		-c & a
	\end{pmatrix} \text{ resp. } \dfrac{1}{ab-\vert c\vert^2}\begin{pmatrix}
		b & -c \\
		-\overline{c} &a
	\end{pmatrix}, \] 
and the map $x_1 e_1 +x_2 e_2 \mapsto x_1 e^{*}_2-x_2  e^{*}_1$ defines a $K$-linear similarity $\tau$ of the Hermitian space $(F, h)$, which maps $L$ onto $\mathfrak{a}L^{*}$. Thus, $(L,h)$ is (anti-)isodual as a rigid analytic space.
\end{remark}\medskip

In connection with conjecture \ref{bc}, it must be noted that the tensor product of two isodual rigid adelic spaces $(E,\sigma)$ and $(F,\tau)$ is itself isodual, the tensor product $\sigma\otimes \tau$ providing a similarity from $E\otimes F$ onto its dual (the same observation holds for anti-isodual spaces). Moreover, the bilinear (or sesquilinear) form \eqref{bsig} satisfies the relation \begin{equation}\label{bst}
	b_{\sigma\otimes \tau} = b_{\sigma}\otimes b_{\tau}.
\end{equation}

If $E$ is any rigid adelic space, the direct product $E \times E^{\vee}$ is both an orthogonal and symplectic isodual rigid adelic space. Indeed, the maps 
\begin{equation*}
\sigma\colon\begin{aligned}[t]
E \times E^{\vee}&\longrightarrow E ^{\vee}\times E \\
(x ,x^{\vee})&\longmapsto (x^{\vee},x).
\end{aligned}
\end{equation*} 
and
\begin{equation*}
\sigma'\colon\begin{aligned}[t]
E \times E^{\vee}&\longrightarrow E ^{\vee}\times E \\
(x ,x^{\vee})&\longmapsto (-x^{\vee},x).
\end{aligned}
\end{equation*} 
are both isometries from $E \times E^{\vee}$ onto its dual, the former being orthogonal, and the latter symplectic. 

Similarly, if $E$ a rigid analytic space over a $CM$ field, the map
\begin{equation*}
	\sigma"\colon\begin{aligned}[t]
		E \times \overline{E^{\vee}}&\longrightarrow \overline{E ^{\vee}}\times E \\
		(x ,x^{\vee})&\longmapsto (x^{\vee},x).
	\end{aligned}
\end{equation*} 
yields an isometry of $E \times \overline{E^{\vee}}$ onto its conjugate dual $\overline{E ^{\vee}}\times E$.

These observations leads to the following proposition :
\begin{proposition}\label{redi}
	Let $F$ be a rigid adelic space ever a number field $K$. The following are equivalent :
	\begin{enumerate}
		\item\label{trois} For all rigid adelic space $E$, one has $H_{min} (E \otimes F)=	H_{min}(E)	H_{min}(F)$.
			\item\label{quatre} For all \emph{isodual} rigid adelic space $E$, one has $H_{min} (E \otimes F)=	H_{min}(E)	H_{min}(F)$.
		\item\label{uno} For all \emph{orthogonal} rigid adelic space $E$, one has $H_{min} (E \otimes F)=	H_{min}(E)	H_{min}(F)$.
		\item\label{deux} For all \emph{symplectic} rigid adelic space $E$, one has $H_{min} (E \otimes F)=	H_{min}(E)	H_{min}(F)$.
\end{enumerate}
If $K$ is a $CM$-field, this is also equivalent to :
\begin{enumerate}[resume]
	\item\label{cinq} For all \emph{unitary} rigid adelic space $E$, one has $H_{min} (E \otimes F)=	H_{min}(E)	H_{min}(F)$.
\end{enumerate}
\end{proposition}
\begin{proof}
	We only have to prove that \ref{uno} $\Rightarrow$ \ref{trois} and \ref{deux} $\Rightarrow$ \ref{trois} as well as \ref{cinq} $\Rightarrow$ \ref{trois}, when $K$ is a $CM$ field.
	Let us check the first implication (the other ones are similar) : assume that $H_{min} (E' \otimes F)=	H_{min}(E')	H_{min}(F)$ holds for all orthogonal rigid adelic space $E'$, and let $E$ be an arbitrary rigid adelic space. We choose $t>0$ such that $H_{min}(E[t]) =H_{min}(E[t]^{\vee})$, that is, thanks to \eqref{hmt},\[  t=\left( \dfrac{H_{min}(E^{\vee})}{H_{min}(E)} \right)^{1/2}. \]
	Since $E[t] \times E[t]^{\vee}$ is orthogonal we have 
	\begin{equation} \label{ein}
		\begin{aligned}
			H_{min} \left( (E[t] \times E[t]^{\vee})\otimes F \right)&=H_{min} (E[t] \times E[t]^{\vee}) H_{min} (F)\\&=H_{min} (E[t]) H_{min} (F) \text{ from lemma \ref{ds}, since }H_{min}(E[t]) =H_{min}(E[t]^{\vee})\\&=tH_{min} (E) H_{min} (F). 
		\end{aligned}
	\end{equation}
On the other hand
\begin{equation} \label{zwei}
	\begin{aligned}
		H_{min} \left( (E[t] \times E[t]^{\vee})\otimes F \right)&=H_{min} (E[t]\otimes F \times E[t]^{\vee}\otimes F )\\&=\min \left( H_{min} (E[t]\otimes F),H_{min} (E[t]^{\vee}\otimes F) \right)  \text{ from lemma \ref{ds}}\\&\leq  H_{min} (E[t]\otimes F)=tH_{min} (E\otimes F).
	\end{aligned}
\end{equation}
Comparing \eqref{ein} and \eqref{zwei} yields $H_{min}(E) H_{min}(F) \leq H_{min}(E\otimes F)$, whence equality, since the reverse inequality is always satisfied.
\end{proof}

As a consequence, in order to prove conjecture \ref{bc}, one can restrict to isodual (resp. orthogonal, resp. symplectic) rigid adelic spaces. In the next sections, we investigate some peculiarities of isodual rigid adelic spaces regarding  stability and tensor multiplicativity.

\section{The Grayson-Stuhler filtration of isodual rigid adelic spaces}\label{SS4}

The GS-filtration of an isodual adelic space has remarkable symmetry properties which rely on the following lemma :
 \begin{lemma}\label{l1}
 	Let $E$ an rigid adelic space with GS-filtration\begin{equation*}\left\lbrace 0 \right\rbrace=E_0 \subset E_1 \subset \dots \subset E_{\ell-1} \subset E_{\ell}=E.\end{equation*}
 	Then the GS-filtration of $E^{\vee}$ is :
 	\[ \left\lbrace 0 \right\rbrace=\left( E_{\ell} \right)^{\perp} \subset \left( E_{\ell-1} \right)^{\perp} \subset \dots \subset \left( E_1 \right)^{\perp} \subset \left( E_0 \right)^{\perp}=E^{\vee}. \]
 \end{lemma}
 \begin{proof}
 	This relies on equation \eqref{dual}, the isometry between $E_{i-1}^{\perp}\slash E_i^{\perp} $ and $\left( E_i\slash E_{i-1} \right)^{\vee}$ and  the observation that $E$ is semistable if and only if $E^{\vee}$ is. Together with proposition \ref{gsc}, this gives the conclusion.
 \end{proof}

This lemma  has the following consequence for isodual spaces :

\begin{proposition}\label{gsi}
	  Let $(E,\sigma)$ be an (anti-)isodual rigid adelic space over a number field $K$, either of \emph{orthogonal}, \emph{unitary} or \emph{symplectic} type, and let
	 \begin{equation*}\left\lbrace 0 \right\rbrace=E_0 \subset E_1 \subset \dots \subset E_{\ell-1} \subset E_{\ell}=E.\end{equation*}
	 be its GS-filtration. Then :
	 \begin{enumerate}
	 	\item for every $0 \leq i \leq \ell$, one has $\sigma E_i =E_{\ell-i}^{\perp}$,
	 	\item  the subpace $E_i$ is totally isotropic with respect to $b_{\sigma}$ if $i \leq \dfrac{\ell}{2}$, and \emph{co}-isotropic if $i \geq \dfrac{\ell}{2}$,
	 	\item if $i \leq \dfrac{\ell}{2}$, the quotient $E_{\ell-i}/E_i$ is (anti-)isodual,
	 	\item if $0\leq i\leq j\leq \dfrac{\ell}{2}$, then $E_j/E_i \times E_{\ell-i}/E_{\ell-j}$ is (anti-)isodual.
	 \end{enumerate}
	 
\end{proposition}
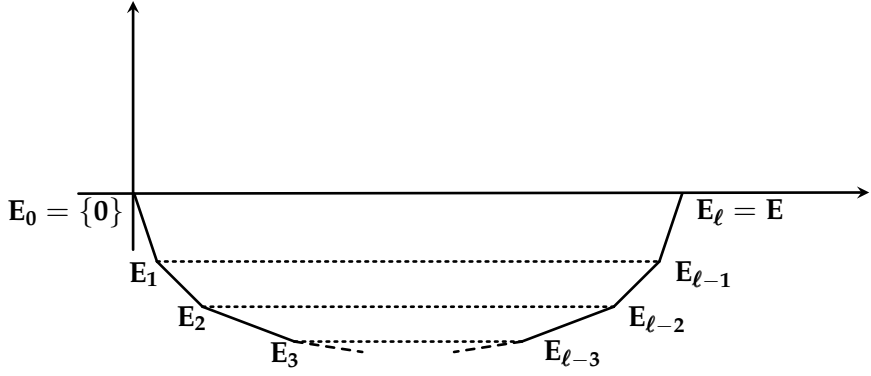
\begin{tikzpicture}[scale=0.3,line cap=round,line join=round,>=triangle 45,x=1.0cm,y=1.0cm]
	\clip(-10.797740707288643,-6.451146896710229) rectangle (30.22574897009442,12.514193281311744);
	\draw [line width=1.pt] (-2.,4.)-- (-1.,1.);
	\draw [line width=1.pt] (-1.,1.)-- (1.,-1.);
	\draw [line width=1.pt] (1.,-1.)-- (5.02998692838312,-2.5392982448377848);
	\draw [line width=1.pt,dash pattern=on 3pt off 3pt] (5.02998692838312,-2.5392982448377848)-- (8.,-3.);
	\draw [line width=1.pt] (22.02459794827468,4.)-- (21.02459794827468,1.);
	\draw [line width=1.pt] (21.02459794827468,1.)-- (19.02459794827468,-1.);
	\draw [line width=1.pt] (19.02459794827468,-1.)-- (14.994611019891558,-2.5392982448377843);
	\draw [line width=1.pt,dash pattern=on 3pt off 3pt] (14.994611019891558,-2.5392982448377843)-- (12.024597948274678,-3.);
	\draw [line width=1.pt,dotted] (-1.,1.)-- (21.02459794827468,1.);
	\draw [line width=1.pt,dotted] (1.,-1.)-- (19.02459794827468,-1.);
	\draw [line width=1.pt,dotted] (5.02998692838312,-2.5392982448377848)-- (14.994611019891558,-2.5392982448377843);
	\draw (-7.996705850226926,4.286153388692981) node[anchor=north west] {$\mathbf{E_0=\lbrace 0 \rbrace}$};
	\draw (-2.6864106003807544,1.4267636387758216) node[anchor=north west] {$\mathbf{E_1}$};
	\draw (-0.5856344575844665,-0.4989478254541019) node[anchor=north west] {$\mathbf{E_2}$};
	\draw (3.4992080422972043,-2.132884825406764) node[anchor=north west]{$\mathbf{E_3}$};
	\draw (15.5,-2.01617503969586) node[anchor=north west] {$\mathbf{E_{\boldell-3}}$};
	\draw (19.196674220413907,-0.4989478254541019) node[anchor=north west] {$\mathbf{E_{\boldell-2}}$};
	\draw (21.239095470354744,1.4851185316312738) node[anchor=north west] {$\mathbf{E_{\boldell-1}}$};
	\draw (22.17277375604198,4.2) node[anchor=north west] {$\mathbf{E_{\boldell}=E}$};
	\draw [-stealth, line width=1.pt] (-2.0463903861114074,1.520733421783512) -- (-2.0463903861114074,12.514193281311744);
	\draw [-stealth, line width=1.pt,domain=-4.4421910718956745:30.22574897009442] plot(\x,{(--90.50041984363673--0.03233718793986995*\x)/22.589192969020008});
\end{tikzpicture}
\begin{proof}
\begin{enumerate}
	\item From Lemma \ref{gssd}, the GS-filtration of $\sigma E$ is 
	\[ \left\lbrace 0 \right\rbrace=\sigma E_0 \subset \sigma E_1 \subset \dots \subset \sigma E_{\ell-1} \subset \sigma E_{\ell}=\sigma E. \]
	The conclusion follows from Lemma \ref{l1}.
	\item One has $\sigma E_i \subset E_{i}^{\perp}$ whenever $2i \leq \ell$ and $\sigma E_i \supset E_{i}^{\perp}$ otherwise, whence the assertion. 
	\item For every $0 \leq i \leq \lfloor \frac{\ell}{2}\rfloor$, the subspace $E_{\ell-i}$ is co-isotropic with respect to $b_{\sigma}$ from the previous assertion.
	Hence $b_{\sigma}$ induces a non degenerate bilinear (resp. sesquilinear) form on the quotient $E_{\ell-i}/E_i$, and consequently $\sigma$ achieves an isometric isomorphism from $E_{\ell-i}/E_i$ onto its (anti-)dual.
	\item Likewise, in the orthogonal and symplectic case, $\sigma$ induces an isometry from  $E_j/E_i$ onto $E_{\ell-j}^{\perp}/E_{\ell-i}^{\perp}\simeq \left( E_{\ell-i}/E_{\ell-j}\right) ^{\vee}$ and from $E_{\ell-i}/E_{\ell-j}$ onto $E_i^{\perp}/E_j^{\perp}\simeq\left(E_j/E_i \right) ^{\vee}$, and similarly with duals replaced by conjugate duals in the unitary case. The conclusion follows.
\end{enumerate}
\end{proof}
The next corollary is an obvious consequence of the second point of the previous proposition, which we state separately because of its importance :
\begin{corollary}\label{mtis}
	If $(E,\sigma)$ is an (anti-)isodual rigid adelic space, either of \emph{orthogonal}, \emph{unitary} or \emph{symplectic} type, which is unstable, then its destabilizing subspace $E_1$ is totally isotropic with respect to $b_{\sigma}$. In particular, $\dim E_1 \leq \dfrac{1}{2}\dim E$.
\end{corollary}
In view of \eqref{bst}, we thus see a connection between conjecture \ref{bc} and the description of  totally isotropic spaces of a tensor product of quadratic, resp. symplectic spaces. \medskip

 Besides $b_{\sigma}$, we may also consider a collection of local bilinear  (resp. sesquilinear) forms  $b_{\sigma_v}$ on $E_v$, defined as follows :
\begin{itemize}
	\item If $v$ is finite or real, we simply set $b_{\sigma_v}(x,y)=\sigma_v(x)(y)$, where $\sigma_v : E_v \rightarrow E_v^{\vee}$ is the map defined by \eqref{sloc}.
	\item If $v$ is a complex place associated to a pair $\left\lbrace \gamma, \overline{\gamma}\right\rbrace$ of complex embeddings , we make use of the two corresponding realizations $\sigma_{\gamma}$ and $\sigma_{\overline{\gamma}}$ of $\sigma_v$, as follows : the map
	 \begin{equation*}
	\begin{array}{rclcc}
	 	E_{\gamma} &\times&E_{\overline{\gamma}}  &\longrightarrow &\mathbf{C}\\
	 \left( 	x\otimes_\gamma\lambda \!\!\!\!\right. &,&\left.\!\! \!\!y \otimes_{\overline{\gamma}} \mu \right)&\longmapsto& \overline{\gamma}(\sigma(x)(y))\overline{\lambda}\mu
	\end{array}
	\end{equation*}
	is left $\mathbf{C}$-antilinear and right $\mathbf{C}$-linear. Choosing a $K$-basis of $E$, we may identify the complex vector spaces $E_{\gamma}=E\otimes_\gamma\mathbf{C}$ and $E_{_{\overline{\gamma}}}=E\otimes_{_{\overline{\gamma}}}\mathbf{C}$ with $\mathbf{C}^d$, where $d=\dim E$, and the above formula induces a sesquilinear map $b_{\sigma_v}$ on $\mathbf{C}^d$ (changing the $K$-basis of $E$ yields an equivalent sesquilinear form on $\mathbf{C}^d$).
	
	\item If $K$ is a $CM$ field and $E$ is anti-isodual via a similarity $\sigma : E \rightarrow\overline{E^{\vee}}$, one obtains $\mathbf{C}$-sesquilinear forms $b_{\sigma_v}$ at infinite places, directly by localizing the $K$-sesquilinear form $b_{\sigma}$, as was done above at real places.
\end{itemize}
Note in particular that if $(E,\sigma)$ is an isodual rigid adelic space of orthogonal type, then $b_{\sigma_v}$ is symmetric if $v$ is either finite or real, and Hermitian if $v$ is complex. Clearly, the destabilizing subspace of an unstable space $E$ is also totally isotropic with respect to $b_{\sigma_v}$ for all $v \in V(K)$. The above corollary \ref{mtis} and its local counterparts thus induce strict restrictions on the GS-filtration of isodual rigid adelic spaces.

Recall that the signature $s(h)$ (resp. $s(b)$) of a non degenerate hermitian or quadratic form $h$ (resp. of its polar form $b$) over an ordered field is the difference $s^{+}(h)-s^{-}(h)$ between the number of positive and negative values taken by $h$ on any orthogonal basis. It is related to the Witt index 
$i(h)$ (common dimension of the  maximal totally isotropic subspaces) by the formula
\begin{equation}\label{ws}
 i(h)= \dfrac{\rk(h)-\vert s(h)\vert}{2}.
\end{equation}
Note also the tensor multiplicativity
\[ s(h\otimes h')=s(h)s(h'). \]
In the sequel, we will say that a non degenerate real quadratic form (resp. complex hermitian form) $h$ of rank $n$ is \emph{definite} if $\vert s(h) \vert =n$, and \emph{Lorentzian} if $\vert s(h) \vert = n-2$. 

\begin{theorem}\label{th1}
	Let $(E,\sigma)$ be either an isodual rigid adelic space of orthogonal type over a number field, or an anti-isodual rigid adelic space of unitary type over a $CM$ field. If $(E,\sigma)$ is unstable, then the dimension of its destabilizing subspace is at most
	\[ \dfrac{1}{2}\left( \dim E - \max_{v\in V_\infty} \vert s(b_{\sigma_v}) \vert\right).\]
In particular, f there exists $v \in V_\infty$ 
such that $b_{\sigma_v}$ is definite, then $E$ is semistable.
\end{theorem}
\begin{proof}
This is an immediate consequence of \eqref{ws} and corollary \ref{mtis}.
\end{proof}
\begin{remark}
Over $\mathbf{Q}$, the rigid adelic spaces satisfying the condition of the above corollary correspond to \emph{unimodular Euclidean lattices over $\mathbf{Z}$}, for which semistability is obvious (see \eg \cite{MR2872959}). Indeed, from Remark \ref{latt}, if $L$ is a Euclidean lattice, with scalar product denoted $x \cdot y$, and $\tau$ is an isometry from $L$ to $L^{\star}=\left\lbrace y \in \mathbf{R} L ,\; \forall x \in L, \; y \cdot x \in\mathbf{Z} \right\rbrace$, then $L$ is orthogonal if and only if $\tau^2=1$. Then it is easily seen that the bilinear form $(x,y)\mapsto \tau(x)\cdot y$ on the space $E=\mathbf{R}L$ cannot be positive definite unless $\tau$ is the identity map.
\end{remark} 

In view of conjecture \ref{bc}, the previous observations lead to the following result
\begin{theorem}\label{sgn}
	Let $(E,\sigma)$ and $(F,\tau)$ be either isodual of orthogonal type over a number field, or anti-isodual of unitary type over a $CM$ field. Suppose there exists an Archimedean place $v$ such that 
	\begin{equation*}
	\vert s(b_{\sigma_v})s(b_{\tau_v})\vert \geq \rank E \rank F - 8.
	\end{equation*}
Then $H_{min}(E \otimes F)=H_{min}(E)H_{min}(F)$.
\end{theorem}
\begin{proof}
	If $E \otimes F$ is semistable, the result is clear. If it is not, then, from Theorem \ref{th1}, the rank of its destabilizing subspace $(E\otimes F)_1$ is at most $4$. Then Theorem B in \cite{MR3035951} implies that $H_r((E\otimes F)_1) \geq H_{min}(E)H_{min}(F)$, whence the result.
\end{proof}

\begin{exs} \label{ex} Let $(E,\sigma)$ and $(F,\tau)$ be as in the previous theorem :
	\begin{enumerate}
	\item\label{un} If there exists an Archimedean place $v$ such that $b_{\sigma_v}$ and $b_{\tau_v}$ are definite, then, $E$, $F$ and $E\otimes F$ are semistable. This extends the result on unimodular Euclidean lattices mentioned in the introduction.
	\item\label{deuxio} Assume that $\rank E \leq 4$ and that there exists a an Archimedean place $v$ $v$ such that $b_{\sigma_v}$ is definite and $b_{\tau_v}$ is Lorentzian. Then $H_{min}(E \otimes F)=H_{min}(E)H_{min}(F)$.
\end{enumerate}
\end{exs}

In the same vein as the previous examples, we get the following :
	\begin{theorem} \label{mix} Let $(E,\sigma)$ and $(F,\tau)$ be either isodual of orthogonal type, or anti-isodual of unitary type over a $CM$ field. 
	Suppose that there exists an Archimedean place $v$ such that one of the following set of conditions is fulfilled :
		\begin{enumerate}
			\item The forms $b_{\sigma_v}$ and $b_{\tau_v}$ are definite.
			\item The form $b_{\sigma_v}$ is definite, the form $b_{\tau_v}$ is Lorentzian, and $F$ is not stable.
			\item The forms $b_{\sigma_v}$ and $b_{\tau_v}$ are Lorentzian, and neither $E$ nor $F$ is stable.
		\end{enumerate}
	Then $H_{min}(E \otimes F)=H_{min}(E)H_{min}(F)$.
\end{theorem}
\begin{proof} The first case is example \ref{ex}\ref{un} above. As for the second case, we let $v$ be an Archimedean place at which $b_{\sigma_v}$ is definite and $b_{\tau_v}$ is Lorentzian. If $F$ is unstable, this implies that its destabilizing subspace $F_1$ is one-dimensional, since it is totally isotropic at $v$. Denoting by $\ell$ the length of the GS-filtration of $F$, we infer hat  $F/F_{\ell-1}$ is also totally isotropic at $v$, whereas the form induced by $b_{\tau_v}$ on $F_{\ell-1}/F_1$ is definite. Clearly, $H_{min}(E\otimes F_1)=H_{min}(E)H_{min}(F_1)$ and  $H_{min}(E\otimes F/F_{\ell-1})=H_{min}(E)H_{min}(F/F_{\ell-1})$, since $F_1$ and $F/F_{\ell-1}$ are one-dimensional. Additionally $H_{min}(E\otimes F_{\ell-1}/F_1)=H_{min}(E)H_{min}(F_{\ell-1}/F_1)$, thanks to Theorem \ref{sgn}. Consequently, one can apply Corollary \ref{c1} and conclude that $H_{min}(E \otimes F)=H_{min}(E)H_{min}(F)$. If $F$ is not stable but semistable, then the same density argument as in the proof of Corollary \ref{r2} applies.
Finally, the third case is an easy combination of the second one with Corollary \ref{c1}.
\end{proof}

\section{Reduction to semistable isodual rigid analytic spaces}\label{sec5}
In this section, we go one step further than Proposition \ref{redi}, and show that the investigation of Conjecture \ref{bc} can be reduced to the case of \emph{semistable} isodual rigid analytic space.
\begin{theorem}\label{redsi}
	The following assertions are equivalent :
	\begin{enumerate}
		\item Conjecture \ref{bc} is true.
		\item\label{dyo} $H_{min} (E \otimes F)=	H_{min}(E)	H_{min}(F)$ whenever $E$ and $F$ are semistable isodual rigid analytic spaces.
	\end{enumerate}
\end{theorem}
\begin{proof}
	The second assertion is obviously implied by the first one, so we only have to prove the reverse implication. We thus assume that the conjecture is proven for the tensor product of two semistable isodual rigid analytic spaces and wish to prove that it is then true for the tensor product of any two rigid analytic spaces $E$ and $F$. Thanks to Proposition \ref{redi}, it is enough to establish the result when $E$ and $F$ are isodual. \smallskip
	
	We proceed in two steps.\medskip
	
	\noindent\textit{\textbf{Step 1.}} We prove the result when $E$ and $F$ are isodual and at least one the two, say $F$, is moreover \emph{semistable}. There is nothing to prove if $E$ is also semistable, so we assume it is not, and let $E_1$ be its destabilizing space. Arguing as in the proof of Proposition \ref{redi}, we can choose $t>0$ such that
	 \[ H_{min}(E_1[t])=H_{min}(E_1[t]^{\vee}), \] so that $E_1[t]\times E_1[t]^{\vee}$ is semistable and isodual. Due to our assumption, since $F$ is also semistable and isodual, we have
	\begin{align}\label{step1}
		H_{min}\left( (E_1[t]\times E_1[t]^{\vee}) \otimes F\right)&=H_{min}(E_1[t]\times E_1[t]^{\vee})H_{min}(F)\nonumber\\
		&=H_{min}(E_1[t])H_{min}(F)\nonumber\\
		&=t H_{min}(E_1)H_{min}(F).
	\end{align}
	On the other hand, 
	\begin{align}
		H_{min}\left( (E_1[t]\times E_1[t]^{\vee})\otimes F \right) &=H_{min} \left( E_1[t] \otimes F \times E_1[t]^{\vee}\otimes F \right)\nonumber\\
		&=\min \left( H_{min} (E_1[t] \otimes F), H_{min} (E_1[t]^{\vee}\otimes F)\right)\nonumber\\
		& \leq H_{min} (E_1[t] \otimes F)= tH_{min} (E_1\otimes F) \leq t H_{min}(E_1)H_{min}(F)
	\end{align}
which together with \eqref{step1} yields  \begin{equation*} 
	H_{min}(E_1)H_{min}(F) = H_{min} (E_1\otimes F).
\end{equation*}
Finally, we obtain
\[ H_{min}(E)H_{min}(F) = H_{min}(E_1)H_{min}(F) =H_{min} (E_1\otimes F) \leq H_{min} (E\otimes F) \leq H_{min}(E)H_{min}(F)\]
and we conclude that $H_{min}(E)H_{min}(F) =H_{min} (E\otimes F)$, as was to be shown.\medskip

\noindent\textit{\textbf{Step 2.}} We now prove that the result remains true for the tensor product of any two isodual rigid adelic spaces $E$ and $F$. Assume, by way of contradiction, that it is not the case, and chose two isodual rigid adelic spaces $E$ and $F$ such that 
	\begin{equation}\label{woc}
		H_{min}(E\otimes F)< H_{min}(E)H_{min}(F).	
	\end{equation} 
From Step 1, we can also assume that neither $E$ nor $F$ is semistable, and we can moreover suppose that $\dim E + \dim F$ is minimal among pairs $(E,F)$ satisfying these properties. 
	Finally, since the previous assumptions are invariant by scaling, we can assume, replacing $F$ by $F[t]$ for a suitable $t$ if necessary, that $H_{min}(F) \leq 1$ (the usefulness of this condition will appear below).
	
	Then, denoting by $\ell$ the length of the GS-filtration of $F$, we infer from Lemma \ref{ds}\ref{ena} that
	\begin{equation}\label{onze}
		H_{min}(E\otimes F_1)\geq H_{min}(E\otimes F_{\ell-1})\geq \min \left( H_{min}(E\otimes F_1),	H_{min}\left( E\otimes (F_{\ell-1}/F_1) \right) \right).
	\end{equation}
	From the minimality assumption on $\dim E + \dim F$, and since $F_{\ell-1}/F_1$ is isodual (Proposition \ref{gsi}), we claim that \[ H_{min}(E\otimes F_{\ell-1}/F_1)=H_{min}(E) H_{min} (F_{\ell-1}/F_1). \] Moreover, the inequality $H_{min}(E\otimes F_1)\leq H_{min}(E)H_{min}(F_1)$ is always satisfied, and we know from the definition of the GS-filtration that $H_{min}(F_1) \leq H_{min} (F_{\ell-1}/F_1)$. Altogether, we can conclude that the right-hand side of \eqref{onze} is equal to $H_{min}(E\otimes F_1)$, so that 
	\begin{equation}\label{treize}
		H_{min}(E\otimes F_{\ell-1})=H_{min}(E\otimes F_1).
	\end{equation}
	Then, using the same argument, we see that
	\begin{align}\label{douze}
		H_{min}(E\otimes F_{\ell-1})\geq H_{min}(E\otimes F)&\geq \min (H_{min}(E\otimes F_{\ell-1}),	H_{min}(E\otimes (F/F_{\ell-1})))\nonumber\\
		&= \min (H_{min}(E\otimes F_1),	H_{min}(E\otimes (F/F_{\ell-1})))\text{ because of \eqref{treize}}\nonumber\\
		&=H_{min}(E\otimes (F_1\times F/F_{\ell-1}))\nonumber\\
		&=H_{min}(E\otimes (F_1\times F_1^{\vee})) \text{ since } F/F_{\ell-1} \simeq F^{\vee}/F_1^{\perp} \simeq F_1^{\vee}\nonumber\\
		&=\min\left( H_{min}(E\otimes F_1),H_{min}(E\otimes F_1^{\vee}) \right).
	\end{align}
	As in the previous step, we can choose $t>0$ such that \[ H_{min}(F_1[t])=H_{min}(F_1[t]^{\vee}) \] so that $F_1[t]\times F_1[t]^{\vee}$ is isodual \emph{and semistable}. From step 1, we thus infer that 
	\begin{align}\label{ap}
		H_{min}(E\otimes (F_1[t]\times F_1[t]^{\vee}))&=H_{min}(E)H_{min}(F_1[t]\times F_1[t]^{\vee})\nonumber\\&=tH_{min}(E)H_{min}(F_1)=t^{-1}H_{min}(E)H_{min}(F_1^{\vee}).
	\end{align}
	On the other hand
	\begin{align}\label{last}
		H_{min}(E\otimes (F_1[t]\times F_1[t]^{\vee}))&=H_{min}(E\otimes F_1[t]\times E\otimes F_1[t]^{\vee})\nonumber\\
		&=\min H_{min}(E\otimes F_1[t], E\otimes F_1[t]^{\vee})\nonumber\\
		&\leq H_{min}(E\otimes F_1[t])=t H_{min}(E\otimes F_1)\leq tH_{min}(E)H_{min}(F_1)
	\end{align}
and similarly
\begin{equation*}
	H_{min}(E\otimes (F_1[t]\times F_1[t]^{\vee})) \leq t^{-1}H_{min}(E\otimes F_1^{\vee})\leq t^{-1}H_{min}(E)H_{min}(F_1^{\vee}).
\end{equation*}
Together with \eqref{ap}, this implies that 
 
	\[ H_{min}(E)H_{min}(F_1)= H_{min}(E\otimes F_1)  \]
	and
	\[ H_{min}(E)H_{min}(F_1^{\vee})=H_{min}(E\otimes F_1^{\vee}). \]

	Consequently, \eqref{douze} reads
	\[  H_{min}(E\otimes F_{\ell-1})\geq H_{min}(E\otimes F)\geq H_{min}(E)\min (H_{min}(F_1), H_{min}(F_1^{\vee})).\]
	As $F_1$ and its dual are semistable, $H_{min}(F_1)=H_r(F_1)=H_r(F_1^{\vee})^{-1}=H_{min}(F_1^{\vee})^{-1}$ and since we have assumed that $H_{min}(F)=H_{min}(F_1)\leq 1$, this implies that  $H_{min}(F_1)\leq H_{min}(F_1^{\vee})$. 
	
The final reformulation of \eqref{douze} is thus
	\[  H_{min}(E\otimes F_{\ell-1})\geq H_{min}(E\otimes F)\geq H_{min}(E)H_{min}(F_1)=H_{min}(E)H_{min}(F).\]
from which we conclude that $H_{min}(E\otimes F)=H_{min}(E)H_{min}(F)$. This contradicts our initial hypothesis.

\end{proof}
\begin{remark}
The theorem and its proof remain true with \emph{isodual} replaced  by \emph{orthogonal}, \emph{symplectic}, or \emph{unitary} (if $K$ is a CM-field). 
\end{remark}

\section{Isoduality and automorphisms}\label{sec4}
The role of automorphisms with respect to the GS-filtration and Conjecture\ref{bc} has been stressed on by several authors (\cite{MR1423622},\cite{MR3096565}, \cite{CN3}, \cite{MR4002393}). We wish to study more specifically in this section its interplay with isoduality.\medskip

An \emph{automorphism} of a rigid adelic space $E$ is an isometry from $E$ to itself, that is, an element of $\GL(E)$ which preserves all local norms $\Vert \cdot \Vert_v$, $v \in V(K)$. If $L$ is the underlying $\mathcal{O}_K$-lattice of the corresponding Hermitian bundle, an automorphism is thus an element of the (discrete) group $\GL(L)$ which simultaneously belongs to the unitary group of every Archimedean completion. With this description, the set $\Aut E$ of such automorphisms is easily seen to be a finite group. It also acts on the dual $E^{\vee}$ by transposition
 \[ g^{\vee} (\varphi):= \varphi \circ g, \ \ g \in \Aut E , \varphi \in E^{\vee}, \] 
and one can identify  $\Aut E^{\vee}$ with the set $\left\lbrace g^{\vee}, g \in \Aut E \right\rbrace$. 

Remarkably, the automorphism group $\Aut E$ stabilizes the Grayson-Stuhler-filtration

\begin{proposition}[\cite{MR1423622}]
Let $\left\lbrace 0 \right\rbrace=E_0 \subset E_1 \subset \dots \subset E_{\ell-1} \subset E_{\ell}=E$ the GS-filtration of a rigid analytic space $E$. Then, $g(E_i)=E_i$ for all $g \in G$ and all $0 \leq 1 \leq i \leq r$.
\end{proposition}

The natural actions of $G=\Aut E$ on $E$ and $E^{\vee}$ described above correspond to faithful representations $\rho : G\rightarrow GL_K(E)$ and  $\rho^{\vee} : G\rightarrow GL_K(E^{\vee})$ :
\begin{align}
	\rho(g)&=g & g \in G \label{rep1}\\
	\rho^{\vee}(g)&= (g^{-1})^{\vee}(\varphi)& g \in G.\label{rep2}
\end{align}

If $(E,\sigma)$ is an isodual  rigid adelic space, there is an additional representation to consider, stemming out from the action of $\sigma$ : for every $g \in \Aut E$, the product $g^{\vee} \sigma$ is an isometry from $E$ to $E^{\vee}$, so that $\sigma^{-1}g^{\vee} \sigma$ is an isometry from $E$ to itself. It follows that the map
\begin{equation}\label{auto}
	g \rightarrow g^{\sigma}:=\sigma^{-1}(g^{-1})^{\vee} \sigma
\end{equation}
is an automorphism of $G=\Aut E$, which gives rise to the "twisted" representation 
\begin{equation}\label{rep3}
	\rho^{\sigma}(g) := \rho (g^{\sigma}), \ \ g \in G.
\end{equation}

Bringing together \ref{rep2}, \ref{auto} and \ref{rep3}, we infer that $\rho^{\sigma}$ and $\rho^{\vee}$ are equivalent as representations of $G$  over $K$. Namely, $\sigma$ induces a $K[G]$-isomorphism from $(E, \rho^{\sigma})$ onto $(E^{\vee},\rho^{\vee})$, which maps the $G$-invariant subspaces of $E$ onto those of $E^{\vee}$ bijectively :

\begin{equation}\label{giso}
 \begin{tikzcd}
 E\arrow[r, "\sigma"] \arrow{d}[swap]{\rho^{\sigma}(g)}	
 & E^{\vee}  \arrow[d,"\rho^{\vee}(g)" ] \\
 E\arrow[r,"\sigma"  ]
 &E^{\vee}  
 \end{tikzcd}
 \end{equation}
The properties of these two representations allow to derive more consequences on the GS-filtration.  
Suppose that the $K[G]$-module $E$ splits as 
\begin{equation*} 
E=\oplus_{i}V_i^{a_i}
\end{equation*}
where the $V_i$s are irreducible pairwise non-isomorphic $K[G]$-modules. If all irreducible components are self-dual, i.e $V_i \sim_{K[G]}V_i^{\vee}$ for all $i$, then clearly, $(E, \rho)$ and $(E^{\vee},\rho^{\vee})$ are also equivalent over $K$.

 The following lemma shows that the above self-duality condition is always satisfied when $K$ is either a totally real or a $CM$ extension of $\mathbf{Q}$.

\begin{lemma}\label{4dot2}
	Let $(E,\sigma)$ be a rigid adelic space over a number field $K$, $G=\Aut E$ its automorphism group. 
	If $K$ is a totally real or $CM$-field, then $E$ and $E^{\vee}$ are isomorphic as $K[G]$-modules, \ie the representations $\rho$ and $\rho^{\vee}$ are equivalent.
\end{lemma}
\begin{proof}
	If $E=\oplus_{i}V_i^{a_i}$ is the splitting of $E$ into irreducible components, it is enough to show that each $V_i$ carries a non-degenerate bilinear (resp. Hermitian) $G$-invariant form if $K$ is a totally real (resp. $CM$) number field. 
	
	Let  $\overline{\phantom{s}}$ stand for the complex conjugation if $K$ is a $CM$-field, or the identity in the totally real case. This extends uniquely to an involution on $K_v$ for each infinite place $v$, which we denote likewise. 
	
	Let $d=\dim E$. If we fix a $K$-basis $\mathcal{B}$ of $E$, and identify the elements of $G$ with their matrices with respect to it, viewed as elements in $M_d(K) \hookrightarrow M_d(K_v)$, we can define

\[ {\mathcal F}(G) := \{ H \in M_d(K) \mid  H = \overline{H}^{tr} 
\mbox{ and }
gH\overline{g}^{tr} = H \mbox{ for all } g\in G \}. \] 

and, for each $v \in V_{\infty}$,	\[ {\mathcal F}_v(G) := \{ 
	H\in M_d(K_v) \mid  
	H = \overline{H}^{tr} \mbox{ and }
	\gamma(g) H \overline{\gamma(g)}^{tr} = H 
	\mbox{ for all } g\in G\}, \]
where $\gamma : K \hookrightarrow K_v$ is an embedding associated to $v$.
Clearly, ${\mathcal F}(G)$ is a finite 
dimensional vector space over the fixed field $K^+$
of $\overline{\phantom{s}}$, and for all $v \in V_{\infty}$, one has \[ \dim _{\R } ({\mathcal F}_v(G)) = \dim _{K^+} ({\mathcal F}(G)). \] Identifying $\R \otimes _{\Q } K^+ $ with $\bigoplus _{v \in V_{\infty} } \R $ we hence get
\[ \oplus _{v \in V_{\infty}} {\mathcal F}_v(G) \cong  \R \otimes _{\Q }  {\mathcal F}(G), \] so that ${\mathcal F}(G) \cong \Q \otimes_{\Q }  {\mathcal F}(G)$ 
is dense in $\oplus _{v \in V_{\infty}} {\mathcal F}_v(G) $. The Gram matrix $H_v$ of $h_v$ with respect to $\mathcal{B}$ belongs to ${\mathcal F}_v(G)$ for all $v \in V_{\infty}$, since $G\leq \Aut(L, h_v) $. Consequently, in a small enough neighborhood of $\left( h_v \right)_{v\in V_{\infty}}$ in $\oplus _{v \in V_{\infty}} {\mathcal F}_v(G)$, one can find a totally positive definite symmetric (resp.Hermitian) matrix $H$ belonging to ${\mathcal F}(G)$. The corresponding quadratic (resp. Hermitian) form, when restricted to $V_i$,  is clearly $G$-invariant and non-degenerate.
\end{proof}

\begin{proposition}\label{autoiso} 
	Let $(E,\sigma)$ be an isodual rigid adelic space over $K$ with GS-filtration \begin{equation*}\left\lbrace 0 \right\rbrace=E_0 \subset E_1 \subset \dots \subset E_{\ell -1} \subset E_\ell=E,
	\end{equation*} and $G=\Aut E$ its automorphism group. We assume that the representation $\rho$ and $\rho^{\vee}$ are equivalent over $K$ (this is the case if $K$ is either totally real or $CM$). Let  $V$ a $K[G]$-submodule of $E_i \slash E_{i-1}$. Then $E_{\ell-i+1}\slash E_{\ell-i}$ contains a $K[G]$-submodule isomorphic to $V^{\vee}$.
\end{proposition}
\begin{proof}
The isometry $\left( E_i \slash E_{i-1} \right)^{\vee} \simeq  E_{i-1}^{\perp} \slash E_i^{\perp}$ is an isomorphism of $K[G]$-modules. On the other hand, the isometry $\sigma$ maps bijectively $ E_i \slash E_{i-1}$ onto $ \sigma(E_i) \slash \sigma(E_{i-1})=E_{\ell-i}^{\perp} \slash E_{\ell-i+1}^{\perp}$ and the latter is $K[G]$-isomorphic to $\left( E_{\ell-i+1}\slash E_{\ell-i} \right)^{\vee}$. As the representations $\rho$, $\rho^{\vee}$ and $\rho^{\sigma}$ are equivalent, we can conclude that $\left( E_{\ell-i+1}\slash E_{\ell-i} \right)^{\vee}$ and $ E_i \slash E_{i-1}$ are $K[G]$-isomorphic, whence the conclusion.
\end{proof}

\begin{corollary}\label{mf}
	Let $(E,\sigma)$ be an isodual  rigid adelic space with automorphism group $G$. Assume that
\begin{enumerate}
	\item $E$ and $E^{\vee}$ are isomorphic as $K[G]$-modules.
	\item The $K[G]$-module $E$ splits as $\oplus_{i=1}^t V_i$, where the $V_i$s are pairwise non isomorphic \emph{absolutely} irreducible $K[G]$-modules.
\end{enumerate}
Then $(E,\sigma)$ is semistable. 
\end{corollary}
\begin{proof}
	The first hypothesis implies that the conditions of Proposition \ref{autoiso} are fulfilled. Consequently, if the length $\ell$ of the filtration were $2$ or more, then any irreducible component $V$ of the destabilizing subspace $E_1$ should also appear as a component of  $E/E_{\ell-1}$, and the multiplicity of $V$ in $E$ would consequently be at least $2$.
\end{proof}
\begin{remark}
	For isodual lattices, in the sense of Remark \ref{latt}, the self-duality condition for irreducible components is automatically satisfied, as  the restriction of the bilinear (resp. sesquilinear) form $h$ to any irreducible component is a nonzero $G$-invariant bilinear form.
\end{remark}

When $E$ is multiplicity-free as a $K[G]$-module, like in the above Corollary, then the tensor multiplicativity $H_{min} (E \otimes F) = H_{min} (E)H_{min} (F)$ holds for any $F$, as was conjectured in \cite{CN3} and fully proven by Rémond \cite[Théorème 1.1]{MR4002393}. From Proposition \ref{autoiso} and its corollary, this situation can hardly occur if $E$ is isodual and unstable : if $E$ is isodual, unstable, and its irreducible components are $K[G]$-isomorphic to their duals, then at least one of those has multiplicity $2$ or more. Thus, a natural "isodual" counterpart of \cite[Théorème 1.1]{MR4002393} should be
\begin{conj}\label{bciso2}
		Let $(E,\sigma)$ be an isodual  rigid adelic space with automorphism group $G$.
		Assume that
		\begin{enumerate}
				\item $E$ and $E^{\vee}$ are isomorphic as $K[G]$-modules.
			\item The $K[G]$-module $E$ admits a decomposition $E=\oplus_{i=1}^t V_i^{a_i}$ into absolutely irreducible $G$-modules with multiplicities $a_i \leq 2$.
		\end{enumerate}
	Then, for all rigid analytic space $F$, one has $H_{min} (E \otimes F) = H_{min} (E)H_{min} (F)$.
\end{conj}

Whether this conjecture is significantly easier than the original conjecture \ref{bc} is unclear. Indeed, if true, this would in particular imply that, over a totally real or $CM$-field, conjecture \ref{bc} is true whenever $E$ has dimension $2$, with no condition on $F$. Notice that the proof of \cite[Théorème 1.1]{MR4002393} relies heavily on the fact that, when $E$ admits a  multiplicity free decomposition $E=\oplus_{i=1}^t V_i$, then, any $G$-invariant subspace of $E\otimes F$ splits as $\bigoplus_{i=1}^rV_i\otimes F_i$, where the $F_i$s are subspaces of $F$ (see \cite[Proposition 2.1]{CN3}). Such a description of the $G$-invariant subspaces of $E\otimes F$ fails to  hold as soon as multiplicities occur.\smallskip

We conclude with a result in the direction of Conjecture \ref{bciso2}, under additional restrictive assumptions.

\begin{proposition}
	Let $(E,\sigma)$ be an isodual  rigid adelic space with automorphism group $G$.
	Assume that
	\begin{enumerate}
		\item $E$ and $E^{\vee}$ are isomorphic as $K[G]$-modules.
		\item The $K[G]$-module $E$ admits a decomposition $E=\oplus_{i=1}^t V_i^{a_i}$ into absolutely irreducible $G$-modules with multiplicities $a_i \leq 2$, and $a_i=2$ for at most one $i$.
	\item $E$ is not stable.
	\end{enumerate}
	Then, for all rigid analytic space $F$, one has $H_{min} (E \otimes F) = H_{min} (E)H_{min} (F)$.
\end{proposition}

\begin{proof}
	The proof is quite similar to that of Corollary \ref{r2}. We may assume that $E$ is unstable, since the result will continue to hold if $E$ is semistable and not stable, by the same continuity argument we used before. Under this assumption, at least one irreducible component has multiplicity greater than $1$, because of Corollary \ref{mf}, which implies that exactly one, say $V_{i_0}$, has multiplicity exactly $2$, because of the second assumption of the proposition. Each absolutely irreducible representation occurring in the decomposition of the destabilizing subspace $E_1$ of $E$ must also occur in that of $E/E_{\ell-1}$, from which we can conclude that $E_1$ is absolutely irreducible and isomorphic to $V_{i_0}$, and  $E/E_{\ell-1}$ as well. In particular, thanks to \cite[Proposition A.3]{MR1423622}, we infer that $H_{min}(E_1\otimes F) = H_{min}(E_1) H_{min}(F)$ and $H_{min}(E/E_{\ell-1}\otimes F) = H_{min}(E/E_{\ell-1}) H_{min}(F)$. Moreover, the quotient $E_{\ell-1}/E_1$ is multiplicity free, so that $H_{min}(E_{\ell-1}/E_1\otimes F) = H_{min}(E_{\ell-1}/E_1) H_{min}(F)$, thanks to \cite[Théorème 1.1]{MR4002393}. Finally, we can apply Corollary\ref{c1} and conclude that  $H_{min} (E \otimes F) = H_{min} (E)H_{min} (F)$.
\end{proof}

\providecommand{\bysame}{\leavevmode\hbox to3em{\hrulefill}\thinspace}
\providecommand{\MR}{\relax\ifhmode\unskip\space\fi MR }
\providecommand{\MRhref}[2]{%
	\href{http://www.ams.org/mathscinet-getitem?mr=#1}{#2}
}
\providecommand{\href}[2]{#2}

\end{document}